\documentclass[11pt]{article}

\usepackage{graphicx,color}%
\usepackage{amsmath,amssymb,amsthm,upref,bm}%
\usepackage{labelfig}%

\newtheorem{theorem}{Theorem}%
\newtheorem{proposition}{Proposition}%
\newtheorem{definition}{Definition}%
\newtheorem{remark}{Remark}

\begin{document}

\baselineskip=4.6mm plus0.1mm minus0.01mm

\makeatletter

\newcommand{\E}{\mathrm{e}\kern0.2pt} 
\newcommand{\D}{\mathrm{d}\kern0.2pt}
\newcommand{\RR}{\mathbb{R}}
\newcommand{\CC}{\mathbb{C}}%
\newcommand{\ii}{\kern0.05em\mathrm{i}\kern0.05em}

\renewcommand{\Re}{\mathrm{Re}} 
\renewcommand{\Im}{\mathrm{Im}}

\def\checkadj{%
\setbox0=\hbox to \textwidth{\hfill\raise0.6ex\hbox{{%
\vtop{\tolerance=7000\hyphenpenalty=3000\leftskip=0pt\rightskip=0pt plus 1pt\relax
\hsize=13mm\tiny\noindent \textcolor{red}{\huge$\checkmark$}}}}\kern-16mm}%
\dp0=0pt\ht0=0pt\vadjust{\box0}}%
\def\alert#1{\textcolor{red}{#1}\checkadj}

\def\bottomfraction{0.9}

\title{\bf Two-dimensional sloshing: \\ domains with interior `high spots'}

\author{Nikolay Kuznetsov and Oleg Motygin}

\date{}

\maketitle

\vspace{-8mm}

\begin{center}
Laboratory for Mathematical Modelling of Wave Phenomena, \\ Institute for Problems
in Mechanical Engineering, Russian Academy of Sciences \\ V.O., Bol'shoy pr. 61, St
Petersburg 199178, Russian Federation \\ E-mail: nikolay.g.kuznetsov@gmail.com;
o.v.motygin@gmail.com
\end{center}

\begin{abstract}
\noindent Considering the two-dimensional sloshing problem, our main focus is to
construct domains with interior high spots; that is, points, where the free surface
elevation for the fundamental eigenmode attains its critical values. The so-called
semi-inverse procedure is applied for this purpose. The existence of high spots is
proved rigorously for some domains. Many of the constructed domains have multiple
interior high spots and all of them are bulbous at least on one side.
\end{abstract}

\setcounter{equation}{0}


\section{Introduction}

The sloshing problem (it describes natural frequencies and the corresponding wave
eigenmodes in an inviscid, incompressible, heavy fluid bounded above by a restricted
free surface) is of great interest to engineers, physicists and mathematicians. Its
two-dimensional version concerns transversal free oscillations of fluid in an
infinitely long canal of uniform cross-section. A historical review of studies in
this area going back to the 18th century can be found in \cite{FK}, whereas various
aspects of the problem are presented in the monographs \cite{FT}, \cite{I} and
\cite{KK}; the last one provides an advanced mathematical approach to the problem
based on spectral theory of operators in a Hilbert space.

Our aim is to consider time-harmonic sloshing in two dimensions and to show that there exist fluid
domains in which `high spots' (points, where the free surface elevation of a
fundamental sloshing mode attains its extrema) are located inside the mean free
surface. It should be recalled that the free surface elevation of a fluid in the sloshing
motion is proportional to the trace of the velocity potential on this horizontal part of domain's boundary. Therefore, at every moment high spots are determined by the trace's maxima and minima provided a time-harmonic
factor is removed.

The notion of high spot was introduced in \cite{KK1} by analogy with
a hot spot; a conjecture about the latter (formulated by J.~Rauch in 1974; see,
e.g., the extensive article \cite{BB}) says that any eigenfunction, corresponding to
the smallest nonzero eigenvalue of the Neumann Laplacian in a bounded domain $F
\subset \RR^m$, attains its maximum and minimum values on $\partial F$. The latter has been proven for some domains, most of which are
special domains in the plane; there is also a counterexample to the conjecture
involving $F \subset \RR^2$ with two holes. For highlights of these results, see the
recent article \cite{JM}.

It is well-known that the problem, that describes sloshing in a three-dimensional
container having a constant depth and vertical walls, is equivalent to the
eigenvalue problem for the Neumann Laplacian in the domain $F \subset \RR^2$, which
is the container's free surface (indeed, separation of variables immediately yields
this). Hence, the absence of interior high spots for such a container follows from the hot spot
result for $F$. 

For other container's geometries, the situation with high spots is not as
simple and not directly connected to hot spots of the free surface. Indeed, it was proved rigorously \cite{KuK} and
established experimentally \cite[Fig.~7]{Ku}, that locations of high spots may vary essentially
for convex axisymmetric containers when their form changes slightly. High spots
are located inside the free surface of bulbous containers (see \cite[Fig.~2]{Ku})
and on the boundary of the free surface when the container's cross-section decreases
with the depth or is constant (see \cite[Fig.~1]{Ku}); see also \cite{KKS} for
further details. The latter (boundary location of high spots) is also proved rigorously
\cite{KK2} for sloshing in troughs of uniform cross-section, whose bottom\,---\,the
graph of a \mbox{$C^2$-function}\,---\,forms nonzero angles with the undisturbed free surface.
Earlier, the same statement was proved for sloshing in two dimensions \cite{KK1}.




Recently, the first study \cite{WTHO} was published which investigates the effects
of surface tension on the location of high spots in two- and three-dimensional
ice-fishing problems. Computational methods are used to demonstrate that the high
spot is in the interior of the free surface for large Bond numbers, but for
sufficiently small Bond number the high spot is on the boundary of the free surface.

Along with the classical linear sloshing problem considered in the present paper,
there are various approaches to nonlinear sloshing phenomena; see, for example, the
monograph \cite{Luk} and references cited in its Chapter~8, in which some typical
nonlinear phenomena discovered experimentally long ago are listed; in particular,
dependence of the sloshing frequency on wave amplitude, mobility of nodal curves on
the free surface, {\it etc}. In this chapter, special attention is paid to sloshing
in cylindrical containers, which allowed the author to obtain some qualitative
results.

Another phenomenon similar to nonlinear sloshing is that of standing waves. Their existence
on the surface of an infinitely deep perfect fluid under gravity is established in
the extensive article \cite{IPT}, where the two-dimensional waves periodic both in space and time are
studied in the framework of fully nonlinear model. Also, several approximate
approaches developed earlier by a number of authors are reviewed.

\clubpenalty=10000
The plan of the present paper is as follows. Statement of the sloshing problem is
formulated in Sect.~2, where general results about it (including the properties of
nodal domains) are also described. Rigorously proved results are presented in
Sect.~3; the so-called semi-inverse procedure is applied for this purpose. An
example of domain with a single interior high spot is investigated in detail in
Sect.~3.1, whereas similar results are outlined for another domain in Sect.~3.2.
Further examples are obtained numerically and considered in Sect.~4. In Sect.~5,
some characteristic features of domains with interior high spots are summarized.

\section{Statement of the problem and general results}

Let an inviscid, incompressible, heavy fluid occupy an infinitely long canal of
uniform cross-section bounded above by a free surface of finite width. The surface
tension is neglected and the fluid motion is assumed to be irrotational and of
small-amplitude. The latter assumption allows us to linearize boundary conditions on
the free surface; see \cite[Sect.~I.1]{M} for details briefly outlined below. In the
case of the two-dimensional motion in planes normal to the generators of canal's
bottom, the following relations arise:
\begin{equation}
g \eta (x, t) + \phi_t (x, 0, t) = 0 \, , \quad \eta_t (x, t) - \phi_y (x, 0, t) = 0
\, . \label{time}
\end{equation} 
In these relations, $\eta (x, t)$ and $\phi (x, y, t)$ are the time-dependent free
surface profile and velocity potential, respectively, whereas rectangular Cartesian
coordinates $(x,y)$ are taken in the plane of motion so that the $x$-axis lies in
the mean free surface, whereas the $y$-axis is directed upwards.

The second condition \eqref{time} is a kinematic condition; it is a consequence of
continuity of the fluid motion and the assumption that the latter is irrotational.
The first condition \eqref{time} follows from the Hamilton's principle
\[ \delta L = 0 \, , \quad \mbox{where} \ \ L = \frac{\rho}{2} \int_0^t \left[ \int_W
|\nabla \phi|^2 \, \D x \D y - g \int_F \eta^2 \, \D x \right] \D t
\]
is the standard Lagrangian in which $\rho$ stands for the fluid's density. Here the
cross-section $W$ of the canal is a bounded simply connected domain whose piecewise
smooth boundary $\partial W$ has no cusps. One of the open arcs forming $\partial W$
is an interval~$F$ of the $x$-axis (the free surface of fluid in equilibrium).

In view of relations \eqref{time}, if time-harmonic oscillations at the radian
frequency $\omega$ having the form
\begin{equation}
\eta (x, t) = \zeta (x) \sin \omega (t-a) \, , \quad \phi (x, y, t) = u (x, y) \cos
\omega (t-a) \label{time'}
\end{equation}
exist for some real $a$, then the real-valued velocity potential $u (x,y)$ must
satisfy the following boundary value problem:
\begin{align}
& u_{xx} + u_{yy} = 0\quad {\rm in}\ W, \label{lap} \\ & u_y =
\nu u\quad {\rm on}\ F, \label{nu} \\ & \partial u/\partial n =
0\quad {\rm on}\ B. \label{nc}
\end{align}
Here the bottom $B = \partial W\setminus \overline F$ is the union of open arcs
lying in the half-plane $y<0$ and complemented by corner points (if there are any)
connecting these arcs, where the normal derivative $\partial/\partial n$ is defined.

The Laplace equation \eqref{lap} expresses the zero-divergence condition for the
velocity; \eqref{nc} is the no-flow condition on the rigid bottom. Relation
\eqref{nu} arises by excluding $\eta$ from relations \eqref{time}, whereas the free
surface elevation $\zeta (x)$ is proportional to the trace $u (x, 0)$ of the
velocity potential according to \eqref{time'}.

We suppose the orthogonality condition
\begin{equation}
\int_F u\,\D x = 0 \label{ort}
\end{equation}
to hold, thus excluding the eigenvalue $\nu=0$ of problem \eqref{lap}--\eqref{nc},
and so the spectral parameter is $\nu = \omega^2/g$, where $g$ is the constant
acceleration due to gravity.

It is known since the 1950s that problem \eqref{lap}--\eqref{ort} has a discrete
spectrum; that is, there exists a sequence of eigenvalues
\begin{equation}
0<\nu_1 \leq \nu_2 \leq \dots \leq \nu_n \leq \dots,
\label{seq}
\end{equation}
each counted according to its finite multiplicity and such that $\nu_n \to \infty$
as $n \to \infty$. The set of corresponding eigenfunctions $\{ u_n \}_1^\infty$
belongs to the Sobolev space $H^1 (W)$ and forms a complete system in an appropriate
Hilbert space. These results can be found in many sources; see, for example,
\cite{KK} for a comprehensive mathematical treatment.

\subsection{Nodal domains and their properties}

Let $N(u) = \{ (x,y)\in \overline W:\, u(x,y)=0 \}$ be the set of nodal lines of a
sloshing eigenfunction $u$. A connected component of $W\setminus N(u)$ is called a
nodal domain for $u$. Properties of the nodal lined and domains are closely related to
our considerations, and so we provide a summary of assertions proved in \cite{KKM}:

\vspace{1mm}

\noindent (i) {\it If\/ $R$ is a nodal domain of\/ $u$, then\/ $\overline R \cap F$ is a
subinterval of\/ $F$.}

\noindent (ii) {\it The number of nodal domains corresponding to\/ $u_n$ is less than
or equal to\/ $n+1$.}

\noindent (iii) {\it The sloshing eigenfunction\/ $u_n$ cannot change sign more than\/
$2n$ times on\/ $F$.}

\vspace{1mm}

\noindent Combining these properties and condition \eqref{ort}, one arrive at the
following statement.
\begin{proposition}\label{prop:1}
A fundamental sloshing eigenfunction $u_1$ has a single nodal line which divides $W$
into two nodal domains; this line has one or both ends on $F$.
\end{proposition}

\subsection{Sloshing in terms of the stream function; auxiliary results}

The approach applied in the paper \cite{KKM} demonstrates that another spectral
problem equivalent to \eqref{lap}--\eqref{ort} is convenient to deal with; it
involves the stream function $v$ (a~harmonic conjugate of $u$ in $W$ defined up to
an additive constant):
\begin{align}
& v_{xx} + v_{yy} = 0 \quad {\rm in} \ W , \label{lapv} \\ & -v_{xx} = \nu v_y
\quad {\rm on} \ F , \label{nuv} \\ & v = 0 \quad {\rm on} \ B . \label{dcv}
\end{align}
Notice that condition \eqref{dcv} is obtained from \eqref{nc} with an
appropriate choice of the additive constant; moreover, it implies both conditions
\eqref{nc} and \eqref{ort}. It is obvious that all eigenvalues of problems
\eqref{lapv}--\eqref{dcv} and \eqref{lap}--\eqref{ort} have the same multiplicity.

Let $N (v) = \{ (x,y) \in \overline W: \, v (x,y) = 0 \}$ denote the set of nodal
lines of a sloshing eigenfunction $v$. A connected component of $W \setminus N$ is
called a nodal domain of $v$. The following results obtained in \cite{KKM} are of importance for our considerations.

\begin{proposition}\label{prop:2}

Let $v$ be a stream eigenfunction in $W$ corresponding to the eigenvalue $\nu_1$, then


{\rm (i)} the single nodal domain of\/ $v$ is $W$;

{\rm (ii)} the trace\/ $v (x, 0)$ cannot change sign on\/ $F$ and has a single extremum
there.
\end{proposition}

Notice that the first assertion is analogous to the Courant nodal domain theorem for
the Dirichlet Laplacian.

\section{Fluid domains with interior high spots \\ (rigorous results)}

A version of the so-called inverse method is applied here. It is worth mentioning
that this method was widely used in continuum mechanics in the pre-computer era; see
\cite{Nem} for a survey. There are two forms of this method that distinguish by the
use of boundary conditions. The method is referred to as semi-inverse if some of
these conditions, but not all of them, are prescribed at the outset which is
convenient for applications in the linear water-wave theory. In particular, Troesch
\cite{Troesch} sought a solution of the sloshing problem in the form of a
combination of harmonic polynomials satisfying the free-surface boundary condition
for a certain frequency. Then the homogeneous Neumann condition was applied for
determining the shape of container's bottom, thus giving a family of paraboloids of
revolution as bottom surfaces. A similar procedure is applied below for construction
of fluid domains with interior high spots.

\subsection{Semi-inverse method for $\bm{\nu = 3/2}$}

A particular pair of conjugate harmonic  velocity potential/stream functions is used in our version of
semi-inverse method, namely:
\begin{align}
 u (x,y) &= \int_0^\infty \frac{\cos k (x-\pi) + \cos k (x+\pi)} {k-\nu} \, \E^{k
y}\,\D k \, , \label{uu} \\ 
 v (x,y) &= \int_0^\infty \frac{\sin k(x-\pi) + \sin
k(x+\pi)} {\nu - k} \, \E^{k y} \, \D k \, . \label{vv}
\end{align}
If $\nu = 3/2$ (this pair and similar ones were introduced in
\cite[Subsection~4.1.1]{KMV}), then both numerators vanish at $k=\nu=3/2$, and so
the integrals are the usual converging infinite integrals. Similar functions were
originally proposed by M.~McIver~\cite{MI}, who used them in her construction of
modes trapped by a pair of two-dimensional bodies in the water wave problem.

It is easy to verify that $u$ and $v$ are conjugate harmonic functions in
$\RR^2_-$, and
\[ u (-x,y) = u (x,y) \quad {\rm and} \quad v (-x,y) = -v (x,y). \]
Moreover, $u$ and $v$ are infinitely smooth up to $\partial \RR_-^2 \setminus \{ x =
\pm \pi,\,y=0 \}$, and well-known facts from theory of distributions imply that
$\left[ u_y - \nu u \right]_{y=0}$ is equal to a linear combination of Dirac's
measures at $x=\pi$ and $x=-\pi$. Therefore,
\begin{equation}
u_y = \nu u\quad {\rm on}\ \partial\RR^2_-\setminus \{ x=\pm \pi,\,y=0 \}.
\label{y=0} 
\end{equation}
The calculated nodal lines of $u$ and $v$ are shown in Fig.~\ref{fig1}\,(b); the
line plotted in solid has the following properties; see \cite[Prop.~2.1]{KKM}.

\begin{proposition}\label{prop:3}
If\/ $\nu = 3/2$ in \eqref{vv}, then along with\/ $\{ x=0,\,y<0 \}$, there is only
one nodal line of\/ $v (x,y)$ in\/ $\RR^2_-$. It is smooth, symmetric about the\/
$y$-axis and its both ends are on the\/ $x$-axis; the right one, say\/ $(x_0,0)$,
lies between the origin and the singularity point\/ $(\pi,0)$.
\end{proposition}

\begin{figure}[t]
\centering \SetLabels \L (-0.07*0.97) {\small (a)} \\ \L (0.055*0.73) {\small
$v(x,0)$} \\ \L (0.055*0.92) {\small $u (x,0)$} \\ \L (0.98*0.798) {\small $x$} \\
\L (0.98*-0.011) {\small $x$} \\ \L (-0.07*0.6) {\small (b)} \\ \L (-0.02*0.52)
{\small $y$} \\ \L (0.38*0.595) {\small $F_{3/2}'$} \\ \L (0.6*0.595) {\small
$F_{3/2}$} \\ \L (0.37*0.21) {\small $W_{3/2}'$} \\ \L (0.58*0.21) {\small
$W_{3/2}$} \\ \L (0.265*0.12) {\small $B_{3/2}'$} \\ \L (0.43*0.46) {\small
$B_{3/2}'$} \\ \L (0.525*0.42) {\small $B_{3/2}$} \\ \L (0.71*0.12) {\small
$B_{3/2}$} \\
\endSetLabels
\leavevmode \strut\AffixLabels{\includegraphics[width=96mm]{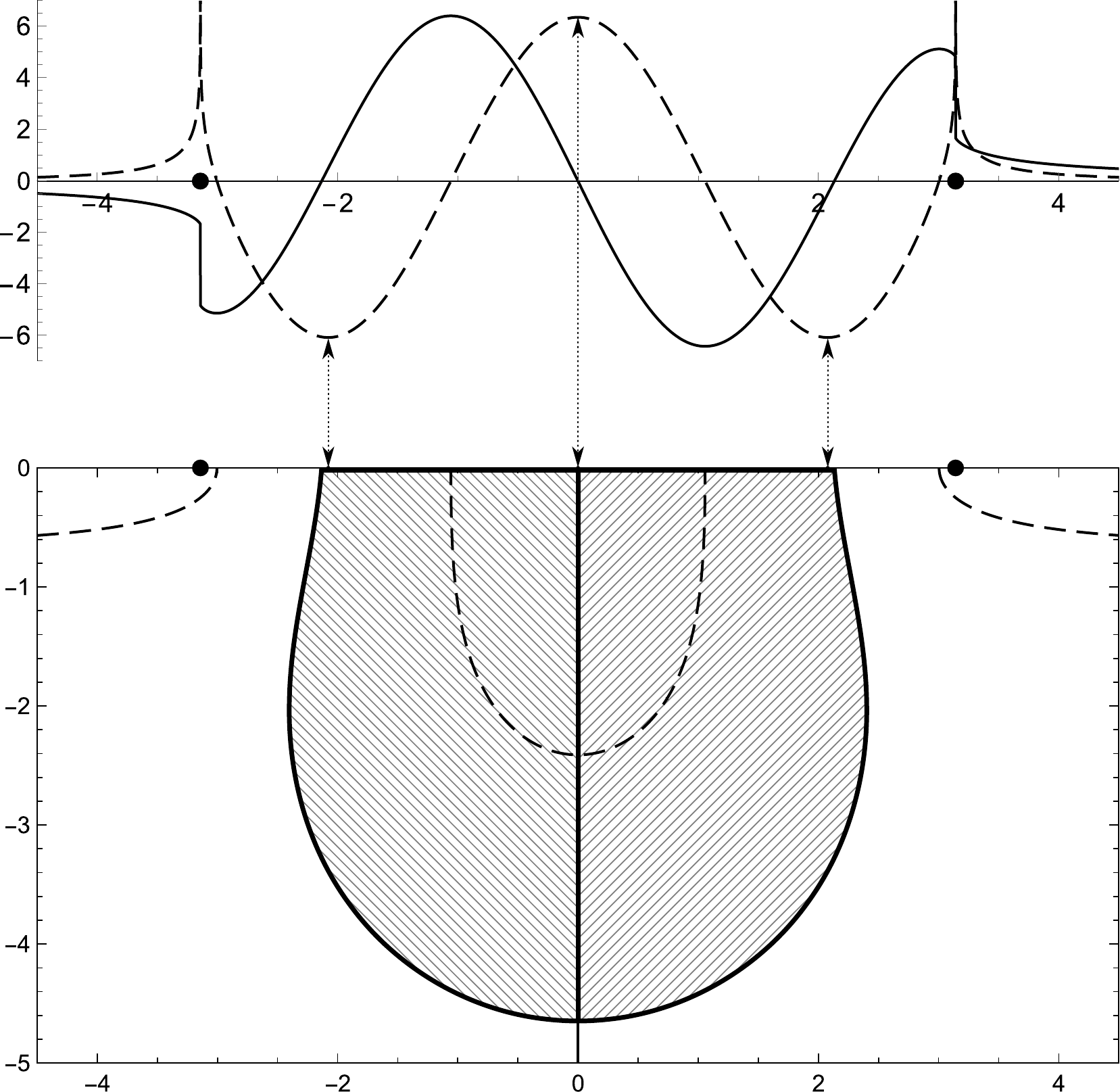}}
\caption{Plotted for $\nu = 3/2$: (a) the traces $u (x, 0)$ (dashed line) and $v (x,
0)$ (solid line); (b) the nodal lines of $u$ (dashed lines) and $v$ (solid lines)
given by \eqref{uu} and \eqref{vv}, respectively. High spots on $\overline{F_{3/2}}$
and $\overline{F_{3/2}'}$ are marked by the arrows connecting them with the extrema
of the velocity potential trace.}
\label{fig1}
\end{figure}
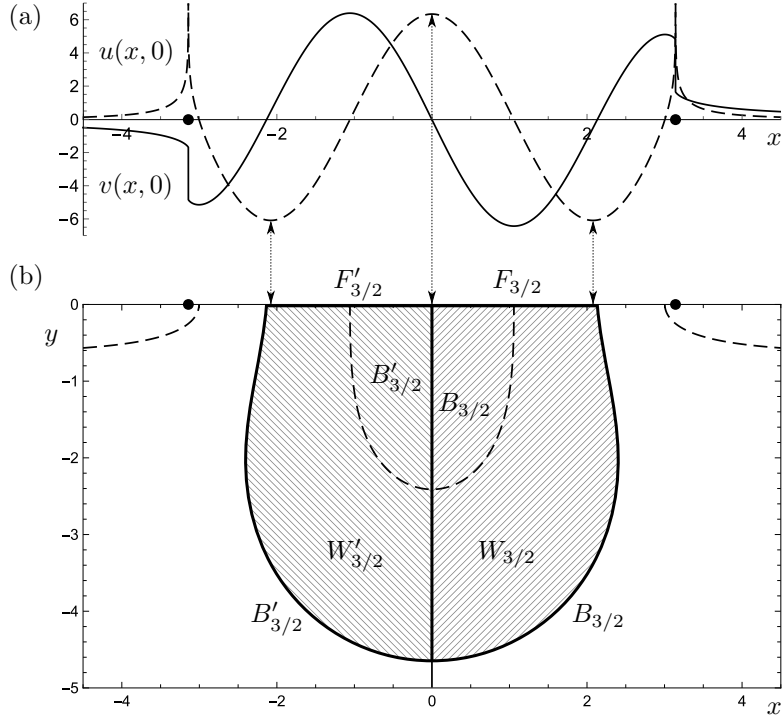


Thus, the right half of the curvilinear nodal line together with a part of the
$y$-axis define the bottom $B_{3/2}$ of a fluid domain; it is denoted $W_{3/2}$ in
Fig.~\ref{fig1}\,(b). Indeed, $v$ given by \eqref{vv} satisfies condition
\eqref{dcv} on this line. On the free surface $F_{3/2}$ of this domain, condition
\eqref{nu} holds for $u$ in view of \eqref{y=0}, and so $u$ and $v$ with $\nu = 3/2$
satisfy the respective versions of the sloshing problem in $W_{3/2}$. 

Notice that we use a half of the domain bounded by the symmetric in $x$ curvilinear
nodal line (let us denote this domain by $\mathring{W}_{3/2}$). The point is that
although the necessary conditions of Proposition~\ref{prop:1} are satisfied for
$\mathring{W}_{3/2}$, the sloshing mode $u$ with $\nu = 3/2$, nevertheless, is not
fundamental in this domain. Indeed, the stream function corresponding to $\nu_1$
cannot have two extrema on the free surface in view of Proposition~\ref{prop:2}. At
the same time, $u$ and $v$ restricted to $W_{3/2}$ satisfy necessary conditions of
Propositions~\ref{prop:1}, \ref{prop:2}, and so they are good candidates for being
the fundamental eigenfunctions corresponding to $\nu = 3/2$.

Indeed, the domain $W_{3/2}$ is nodal for $v$; see the graph of its trace in
Fig.~\ref{fig1}\,(a). Furthermore, the following assertion was proved in
\cite[Th.~2.6]{KKM}.
\begin{theorem}\label{theor1}
In the fluid domain\/ $W_{3/2}$, the sloshing eigenfunction\/ $u$ given by\/
\eqref{uu} with\/ $\nu = 3/2$ has a single nodal line, whose one endpoint is\/
$(x_n,0)$ and\/ $x_n \in (0, x_0)$ (here $x_0$ is the right endpoint of $F_{3/2}$
defined in Proposition~\ref{prop:3}), is the only minimum point of\/ $v (x, 0)$ on\/ $\{ x > 0
\}$, whereas the second endpoint is on the $y$-axis.
\end{theorem}

\begin{remark}\label{rem1}
{\rm The representation
\begin{equation}
u (x,0) = 2 \pi \cos \nu x + \int_0^\infty \left[ \E^{-(\pi -x) k \nu} + \E^{-(\pi +
x) k \nu} \right] \frac{k \, \D k}{1+k^2} \label{9}
\end{equation}
valid for $x \in [0, \pi)$ (see \cite[Form. (2.9)]{KKM}) implies that $u_x (0,0) =
0$ and $u_{xx} (0,0) < 0$. Hence $u (x, 0)$ attains maximum at $x = 0$, and so the
origin is a high spot, but located on the boundary of the free surface $F_{3/2}$.}
\end{remark}

The following assertion implies that there is an interior high spot on $F_{3/2}$; it
corresponds to the right minimum of $u (x, 0)$, which is close to the right endpoint
of the free surface $F_{3/2}$; see Remark~\ref{remark:2} below.

\begin{theorem}\label{theor2}
The sloshing eigenfunction\/ $u$ given by\/ \eqref{uu} with\/ $\nu =
3/2$ has an interior high spot on $F_{3/2}$.
\end{theorem}

\begin{proof}
The representation (see \cite[Form.\ (2.7)]{KKM})
\begin{equation}
v (x,y) = \E^{\nu y} \left[ v (x,0) + 2x \int_y^0 \frac{k^2 -
(\pi^2-x^2)}{[k^2+(\pi-x)^2]\,[k^2+(\pi+x)^2]} \, \E^{-k\nu}\,\D k \right]
\label{10}
\end{equation}
implies that 
\[ v_y (x_0, 0) = u_x (x_0, 0) = \frac{2 \, x_0 \, (\pi^2 - x^2_0)}{(\pi - x_0)^2
\, (\pi + x_0)^2} = \frac{2 \, x_0}{\pi^2 - x^2_0} > 0 \, .
\]
Then $u_x (x_h, 0) = 0$ at some $x_h < x_0$ in view of Remark~\ref{rem1}. Hence, an
interior high spot is located at $(x_h, 0)$ on the left of the endpoint $(x_0, 0)$.
By symmetry, $(-x_h, 0)$ is also an interior high spot located on the right of the
endpoint $(-x_0, 0)$.
\end{proof}

\begin{remark}\label{remark:2}
{\rm According to computations, $x_h \approx 2.077836$, whereas the endpoint of the
free surface $F_{3/2}$ is at $(x_0, 0)$ with $x_0 \approx 2.132704$, that is, the
distance from the high spot to the endpoint is approximately $0.054868$, which is
less than 3\,\% of the distance from the origin to the endpoint of $F_{3/2}$.}
\end{remark}


 
\begin{remark}\label{remark:3}
{\rm The same considerations are valid for the domain $W_{3/2}'$ on the left of the
$y$-axis; see Fig.~\ref{fig1}\,(b). Therefore, $W_{3/2}'$ also provides an example
of domain with an interior high spot.}
\end{remark}

Now we turn to a geometric property intrinsic to domains with interior high spots.
According to the following definition, it is evidently holds for $W_{3/2}$; see
Fig.~\ref{fig1}\,(b).

\begin{definition}
{\rm A fluid domain $W$ satisfies John's condition if it is confined to the strip
bounded by the straight vertical lines through the endpoints of the free
surface~$F$. Domains violating this condition are called bulbous.}
\end{definition}

\begin{proposition}\label{prop:4}
The domain $W_{3/2}$ is bulbous.
\end{proposition}

\begin{proof}
To be specific, let us show that $W_{3/2}$ is bulbous on the right-hand side. Let us
consider $B_{3/2}$ as the graph of the implicit function $x \mapsto y$ defined by
the equation $v (x, y) = 0$ in a neighbourhood of $(x_0, 0)$\,---\,the right
endpoint of $B_{3/2}$. Therefore, to establish that $W_{3/2}$ is bulbous it is
sufficient to prove that 
\[ y'(x_0) = -v_x (x_0, 0)/v_y (x_0, 0) < 0 \, . \]

After some algebra, it follows from \eqref{10} that
\begin{equation*}
y'(x_0) = \frac{x_0^2 - \pi^2}{2 \, x_0} \, v_x (x_0, 0) = \frac{\pi^2 - x_0^2}{2 \,
x_0} \, u_y (x_0, 0) \, ,
\end{equation*}
where the last equality is a consequence of the Cauchy--Riemann equations. In view
of condition \eqref{y=0} we have for $\nu = 3/2$:
\begin{equation}
y' (x_ 0) = \frac{3 (\pi^2 - x_0^2)}{4 \, x_0} \, u(x_0, 0) \, .
\label{eq:y'}
\end{equation}
Since there is only one nodal line of $u$ in $W_{3/2}$, and its right endpoint is
$(x_n,0)$, the trace $u(x,0)$ is negative for $x \in (x_n,x_0]$. Then, \eqref{eq:y'}
implies that $y' (x_0) < 0$, which completes the proof.
\end{proof}

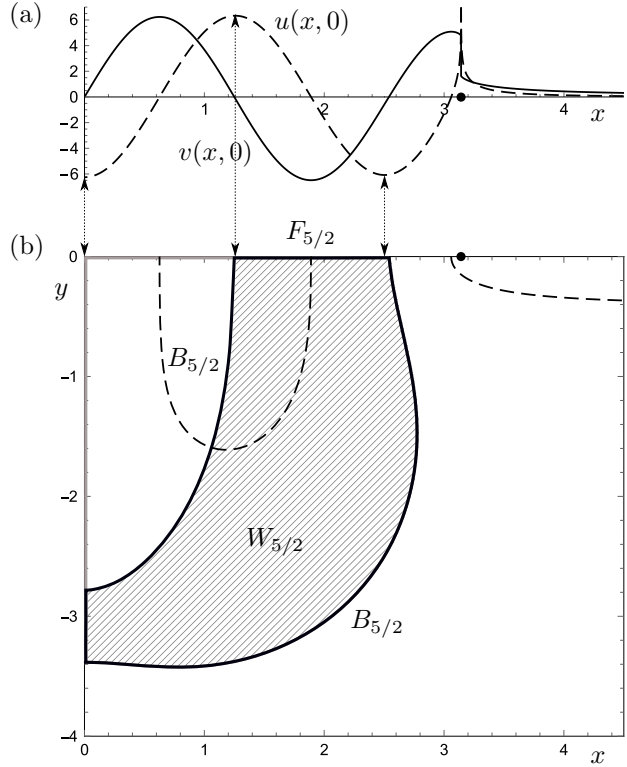
\begin{figure}[t]
\centering \SetLabels \L (-0.09*0.98) {\small (a)} \\ \L (0.21*0.795) {\small
$v(x,0)$} \\ \L (0.38*0.97) {\small $u (x,0)$} \\ \L (0.94*0.847) {\small $x$} \\ \L
(0.94*-0.01) {\small $x$} \\ \L (-0.09*0.67) {\small (b)} \\ \L (-0.01*0.615)
{\small $y$} \\ \L (0.4*0.69) {\small $F_{5/2}$} \\ \L (0.19*0.52) {\small
$B_{5/2}$} \\  \L (0.33*0.28) {\small $W_{5/2}$} \\ \L (0.515*0.17) {\small
$B_{5/2}$} \\
\endSetLabels
\leavevmode \strut\AffixLabels{\includegraphics[width=75mm]{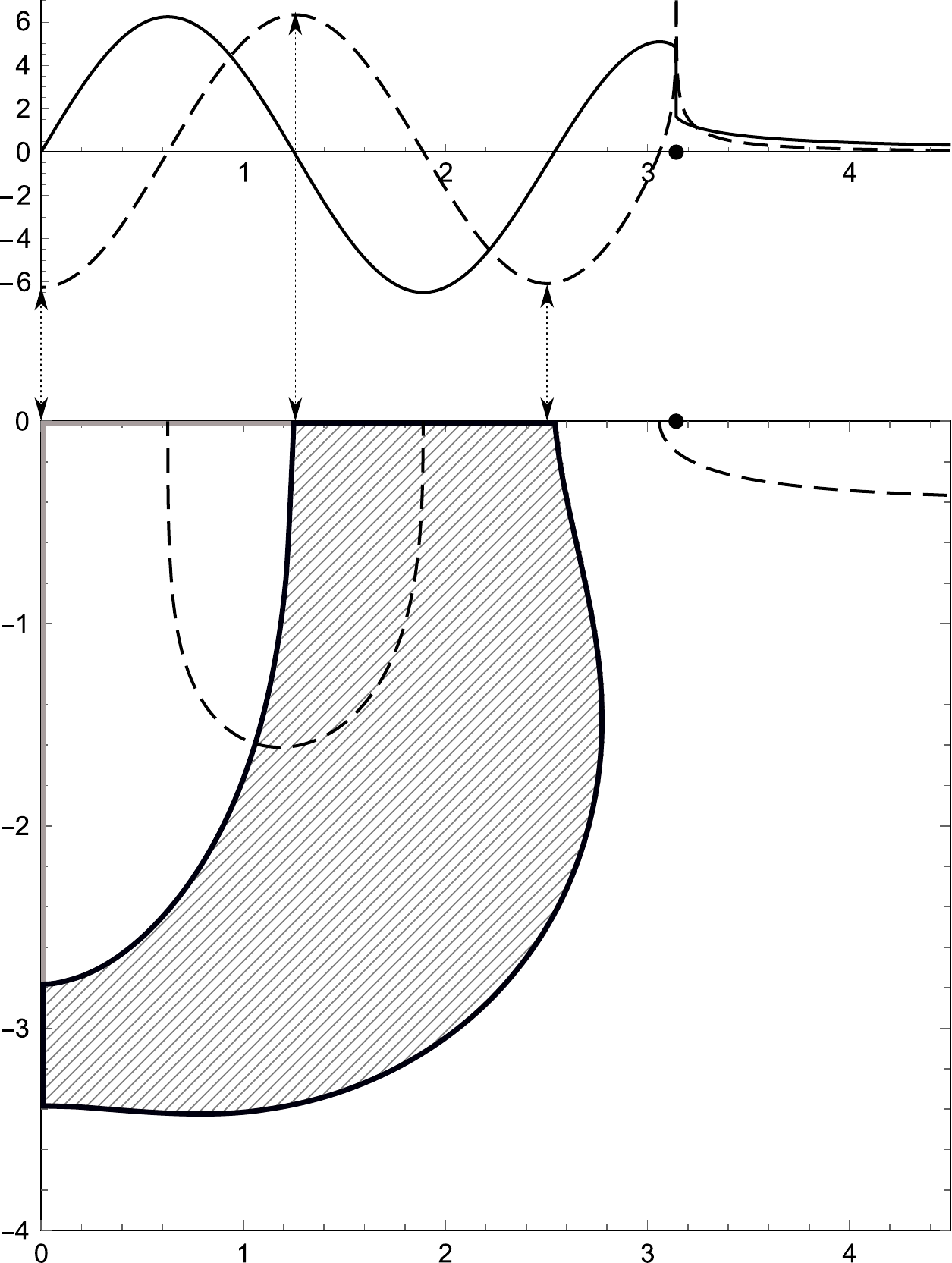}}
\caption{Plotted for $\nu = 5/2$: (a) the traces $u (x, 0)$ (dashed line) and $v (x,
0)$ (solid line); (b) the nodal lines of $u$ (dashed lines) and $v$ (solid lines)
given by \eqref{uu} and \eqref{vv}, respectively. High spots on $F_{5/2}$ are marked
by the arrows connecting them with the extrema of the velocity potential trace.}
\label{fig2.5}
\end{figure}


In conclusion of this section, it is worth mentioning an essential property of the
harmonic function $v$ defining the domain $W_{3/2}$.\,\ This function has a
stagnation point in $\RR^2_-$, namely, the point, where the curvilinear part of
$B_{3/2}$ meets the negative \mbox{$y$-axis}. Indeed, the derivatives along both
these lines vanish because they are nodal lines of $v$. In what follows, the same
property\,---\,the presence of a stagnation point\,---\,will be used in more
complicated cases.

\begin{remark}\label{remark:4}{\rm
It should be said, that the presence of stagnation points of~$v$ and, as a
consequence, of corner points on boundaries of all domains constructed in the paper,
is rather a matter of convenience for our considerations. For instance, it is easy
to observe that any line $v=-c$ in Fig.~\ref{fig2.5} is located inside $W_{3/2}$
provided $0<c<-\min\{v(x,0):x\geq0\}$, and the domain bounded by this line, say
$W^c_{3/2}$, has a smooth bottom. However, $W^c_{3/2}$ can be used as a sloshing
domain instead of~$W_{3/2}$; moreover, the extrema of $u$ are still located inside
the free surface $F^c_{3/2}$ when $c$ is sufficiently small.}
\end{remark}

\subsection{The case $\bm{\nu = 5/2}$; multiple interior high spots}

Rigorous considerations, analogous to those in Sect.\ 3.1, are applicable to the
case when $u$ and $v$ are given by \eqref{uu} and \eqref{vv}, respectively, but with
$\nu = 5/2$. It occurs that the domain similar to $W_{3/2}$ (that is, adjacent to
the $y$-axis), in which this~$u$ is an eigenfunction corresponding to
$\nu = 5/2$, has no interior high spot. Indeed, it satisfies John's condition which
guarantees \cite{KK1} that high spots are on the boundary of the free surface; see
Fig.~\ref{fig2.5}\,(b), where this domain left blank. However, this $u$ is also an eigenfunction corresponding to $\nu = 5/2$ in the domain denoted
by~$W_{5/2}$ in Fig.~\ref{fig2.5}\,(b). This follows from Proposition~\ref{prop:2} because this
domain is nodal for $v$; indeed, the negative $y$-axis is a nodal line of $v$, as
well as both lines marked $B_{5/2}$\,---\,the two curved parts forming the bottom of
$W_{5/2}$.

There are two interior high spots on the free surface $F_{5/2}$; one of them
corresponds to the maximum of $u (x, 0)$ attained at $x \approx 1.257429$ between
the endpoints of the nodal line of $u$, which divides $W_{5/2}$ into two nodal
domains. It should be noted that the left line $B_{5/2}$ emanates from $x \approx
1.249757$, which is slightly smaller. The second high spot corresponds to the
minimum of $u (x, 0)$ attained at $x \approx 2.503159$, which is very close to the
right endpoint of $F_{5/2}$ at $x \approx 2.539769$.

It is clear that the domain $W_{5/2}$ is bulbous on both sides. It is obvious on the
left-hand side, whereas for the right-hand side the reasoning similar to the proof
of Proposition~\ref{prop:4} is applicable. Of course, the domain obtained by reflection of
$W_{5/2}$ in the $y$-axis also provides an example of domain with two interior high
spots.

As in the case $\nu = 3/2$, the function $v$ defining the domain $W_{5/2}$ has
stagnation points in $\RR^2_-$. However, there are two of them, namely, the points,
where both curves $B_{5/2}$ meet the negative $y$-axis.

\section{Further examples of domains with multiple interior high spots (numerical
results)}

In this section, we present other domains with multiple interior high spots, but
they are obtained numerically using the following procedure. For a specified
value of $\nu$ the bottom is defined by a nonzero level line of $v$, whose level is
chosen so that this line has two branches crossing transversally at a stagnation
point, thus forming the bottom of the required domain.

This procedure is applicable for integer values of $\nu$ as well. However, the
following functions
\begin{align}
 u (x,y) &= \int_0^\infty \frac{\cos k (x-\pi) - \cos k (x+\pi)} {k-\nu} \, \E^{k
y} \,\D k \, , \label{uu'} \\
 v (x,y) &= \int_0^\infty \frac{\sin k(x-\pi) - \sin k(x+\pi)} {\nu - k} \, \E^{k y} 
\, \D k \label{vv'}
\end{align}
are used in this case instead of \eqref{uu} and \eqref{vv}.

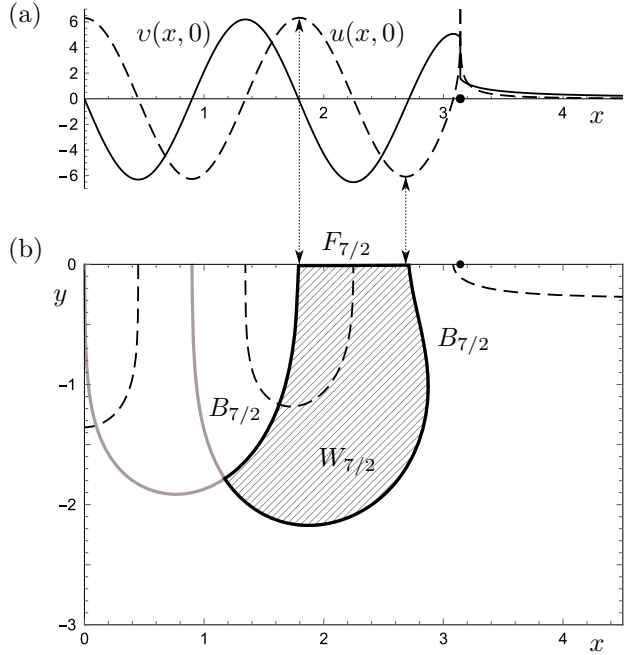
\begin{figure}[t!]
\centering \SetLabels \L (-0.09*0.98) {\small (a)} \\ \L (0.137*0.95) {\small
$\upsilon(x,0)$} \\ \L (0.48*0.95) {\small $u (x,0)$} \\ \L (0.94*0.82) {\small $x$}
\\ \L (0.94*-0.01) {\small $x$} \\ \L (-0.09*0.61) {\small (b)} \\ \L (-0.01*0.54)
{\small $y$} \\ \L (0.465*0.62) {\small $F_{7/2}$} \\ \L (0.46*0.28) {\small
$W_{7/2}$} \\ \L (0.265*0.36) {\small $B_{7/2}$} \\ \L (0.67*0.47) {\small
$B_{7/2}$} \\
\endSetLabels
\leavevmode \strut\AffixLabels{\includegraphics[width=75mm]{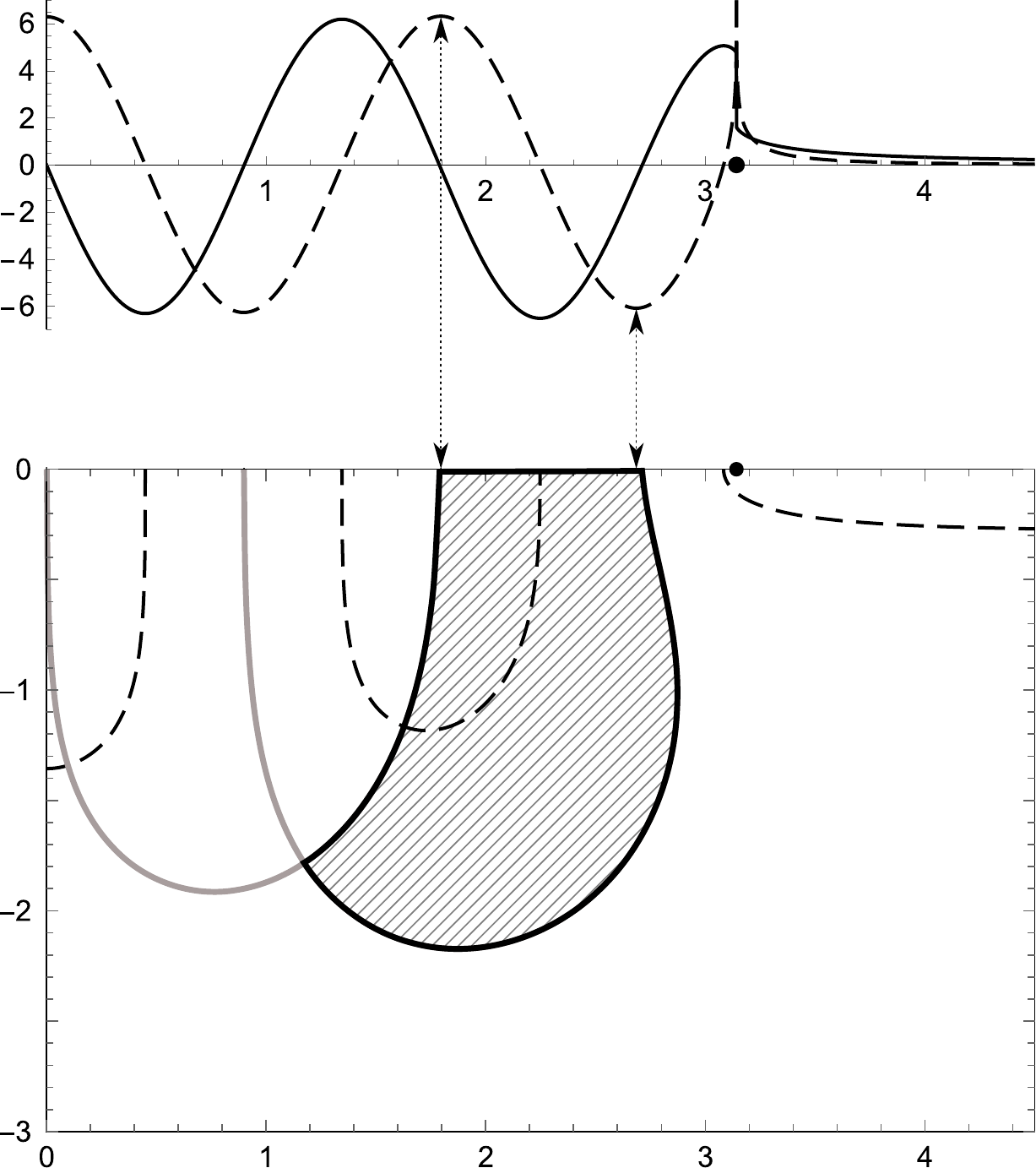}}
\caption{Plotted for $\nu = 7/2$: (a) the traces $u (x, 0)$ (dashed line) and $v (x,
0)$ (solid line); (b) the nodal lines of $u$ (dashed lines), and the level lines $v
\approx -0.023145$ (solid lines). Two interior high spots on $F_{7/2}$ are marked by
arrows connecting them with the corresponding extrema of the velocity potential's
trace.}
\label{fig3.5}
\end{figure}

%

\subsection{The case $\bm{\nu = 7/2}$}

If $v$ is given by \eqref{vv} with $\nu = 7/2$, then a stagnation point occurs at
the level approximately equal to $-0.023145$; see Fig.~\ref{fig3.5}\,(b), where this
point is at the intersection of two solid lines marked $B_{7/2}$. They enclose the
domain $W_{7/2}$, in which $u$ given by \eqref{uu} with $\nu = 7/2$ is an
eigenfunction corresponding to $\nu = 7/2$. Indeed, since $v$ is less than
$-0.023145$ in this domain, it is the nodal one for the stream function equal to the
difference of $v$ and this value.

There are two interior high spots on $F_{7/2}$: near its left endpoint at $x \approx
1.795807$, and close to the right endpoint at $x \approx 2.685549$, within
approximately $0.026076$ from the endpoint. The character of these high spots is the
same as those on $F_{5/2}$. Like $W_{5/2}$, the domain $W_{7/2}$ is bulbous on both
sides. Of course, the domain obtained by reflection of $W_{7/2}$ in the $y$-axis
also provides an example of domain with two interior high spots.

\begin{figure}[t]
\centering \SetLabels \L (-0.09*0.98) {\small (a)} \\ \L (0.07*0.95) {\small
$v(x,0)$} \\ \L (0.44*0.95) {\small $u (x,0)$} \\ \L (0.94*0.82) {\small $x$} \\ \L
(0.94*-0.01) {\small $x$} \\ \L (-0.09*0.61) {\small (b)} \\ \L (-0.01*0.54) {\small
$y$} \\ \L (0.45*0.62) {\small $F_{3}$} \\ \L (0.43*0.28) {\small $W_{3}$} \\ \L
(0.3*0.47) {\small $B_{3}$} \\ \L (0.64*0.46) {\small $B_{3}$} \\
\endSetLabels
\leavevmode \strut\AffixLabels{\includegraphics[width=75mm]{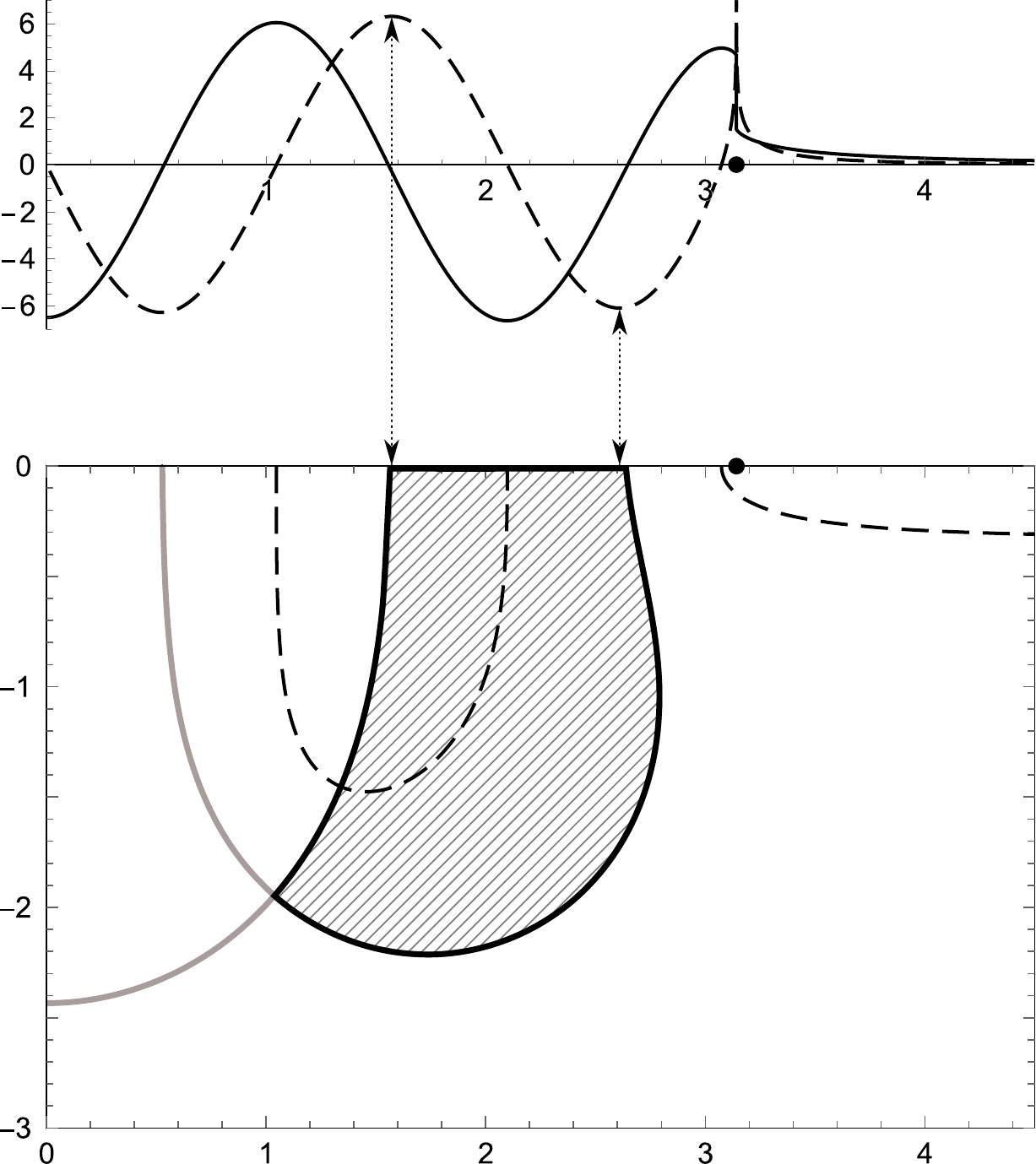}} \caption{Plotted
for $\nu = 3$: (a) the traces $u (x, 0)$ (dashed line) and $v (x, 0)$ (solid line)
given by \eqref{uu} and \eqref{vv}, respectively; (b) the nodal lines of $u$ (dashed
lines), and the level lines $v \approx -0.150899$ (solid and dotted lines). Interior high
spots on $F_{3}$ are marked by arrows connecting them with extrema of the velocity
potential trace.}
\label{fig3}
\end{figure}
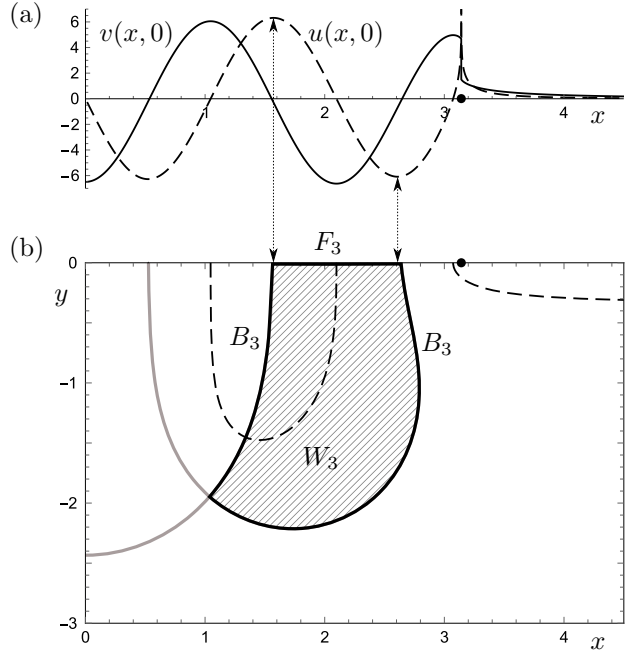


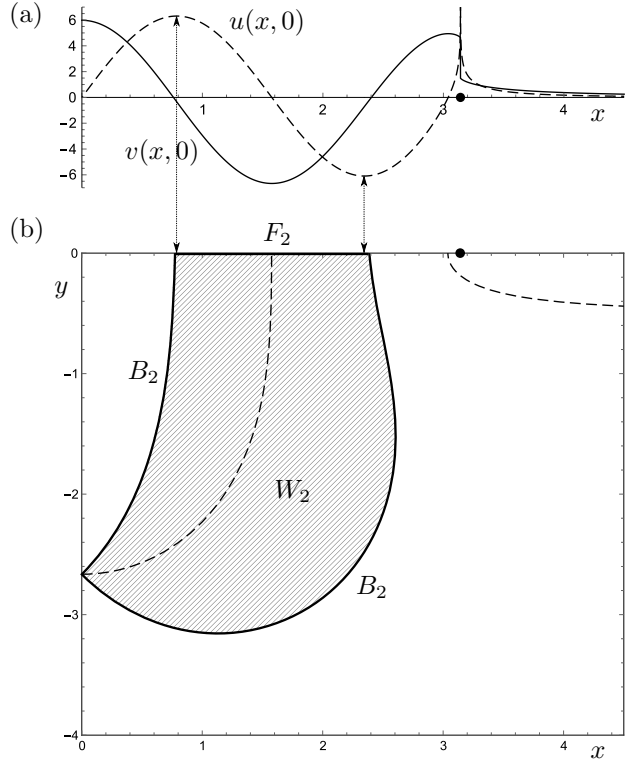
\begin{figure}[t]
\centering \SetLabels \L (-0.09*0.98) {\small (a)} \\ \L (0.115*0.795) {\small
$v(x,0)$} \\ \L (0.3*0.97) {\small $u (x,0)$} \\ \L (0.94*0.847) {\small $x$} \\ \L
(0.94*-0.01) {\small $x$} \\ \L (-0.09*0.69) {\small (b)} \\ \L (-0.01*0.615)
{\small $y$} \\ \L (0.36*0.685) {\small $F_{2}$} \\ \L (0.38*0.34) {\small $W_{2}$}
\\ \L (0.525*0.21) {\small $B_{2}$} \\ \L (0.12*0.5) {\small $B_{2}$} \\
\endSetLabels
\leavevmode \strut\AffixLabels{\includegraphics[width=75mm]{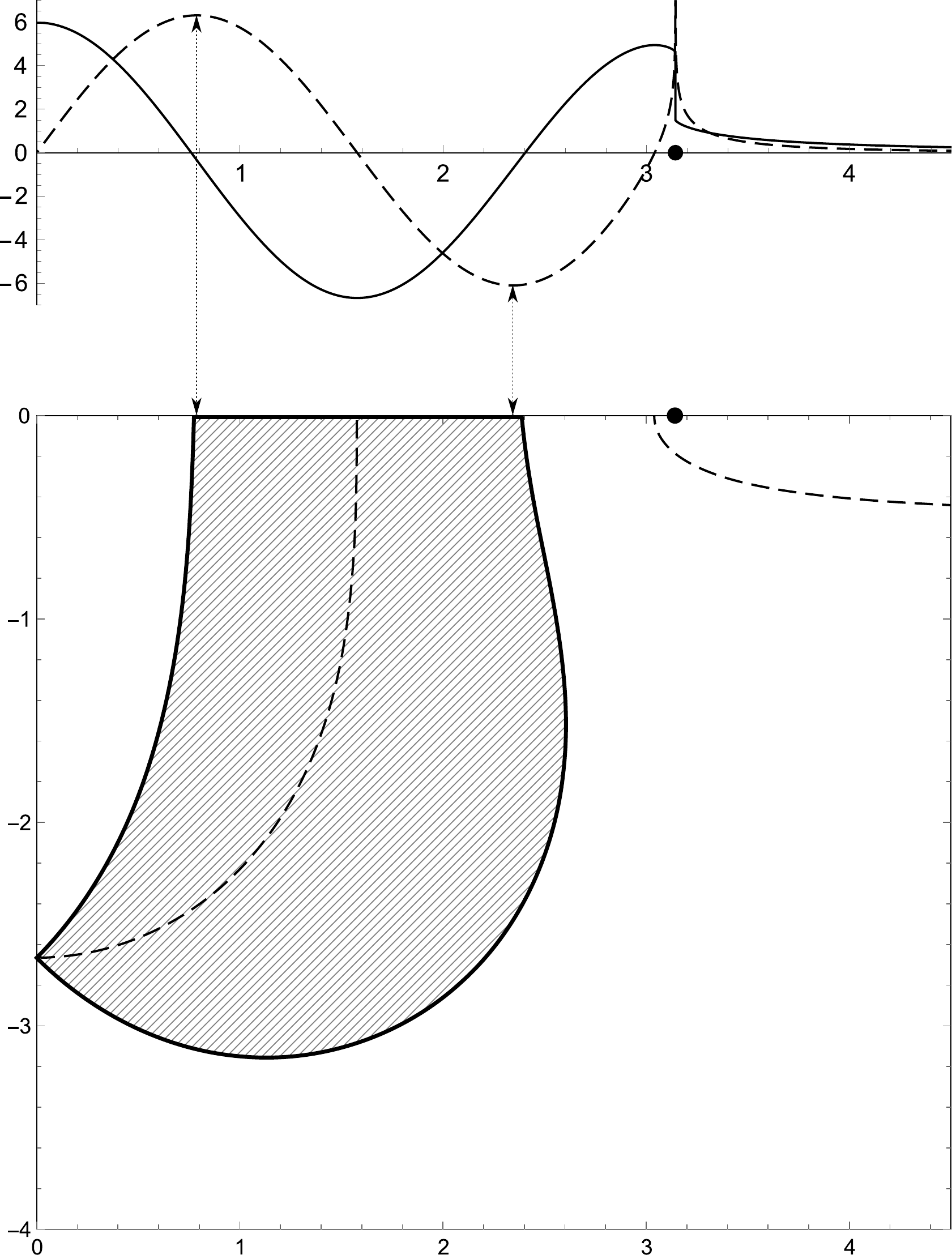}} \caption{Plotted
for $\nu = 2$: (a) the traces $u (x, 0)$ (dashed line) and $v (x, 0)$ (solid line);
(b) the nodal lines of $u$ (dashed lines) and the level line $v \approx -0.185125$
(solid line). Interior high spots on $F_{2}$ are marked by arrows connecting them
with extrema of the velocity potential trace.}
\label{fig2}
\end{figure}


\subsection{The case $\bm{\nu = 3}$}

If $v$ is given by \eqref{vv'} with $\nu = 3$, then a stagnation point occurs at the
level approximately equal to $-0.150899$; see Fig.~\ref{fig3}\,(b), where this point
is at the intersection of two solid lines marked $B_{3}$. They enclose the domain
$W_{3}$, in which $u$ given by \eqref{uu'} with $\nu = 3$ is an
eigenfunction corresponding to $\nu = 3$. Indeed, since $v$ is less than $-0.150899$
in this domain, it is the nodal one for the stream function equal to the difference
of $v$ and this value. 

There are two interior high spots on $F_{3}$: near its left endpoint at $x \approx
1.5715649$, and close to the right endpoint at $x \approx 2.6095109$, within
approximately $0.029250$ from the endpoint. Comparing the fluid domains $W_{3}$ and
$W_{3/2}$ shown in Figs.~\ref{fig3} and \ref{fig3.5}, respectively, we see that they
have the same structure of defining lines and the same number of high spots. Of
course, the domain obtained by reflection of $W_{3}$ in the $y$-axis also provides
an example of domain with two interior high spots.

\subsection{The case $\bm{\nu = 2}$}

Applying the same procedure for this value of $\nu$, one obtains the fluid domain
$W_{2}$ with two interior high spots on $F_{2}$, which distinguishes essentially
from $W_{3/2}$, $W_{5/2}$, $W_{7/2}$ and $W_{3}$; see Fig.~\ref{fig2}\,(b). Namely,
the nodal line of the eigenfunction $u$ (recall that \eqref{uu'} defines
it for $\nu = 2$) connects $F_{2}$ with the stagnation point of $v$ on $B_{2}$.

There are two interior high spots on $F_{2}$, both located near its endpoints: on
the left at $x \approx 0.786780$ and on the right at $x \approx 2.343392$. The
corresponding endpoints of $F_{2}$ are at $x \approx 0.774530$ and at $x \approx
2.387143$, respectively. Finally, we notice that the domain $W_{2}$ is bulbous on
both sides like $W_{5/2}$, $W_{7/2}$ and $W_{3}$. Of course, the domain obtained by
reflection of $W_{2}$ in the $y$-axis also provides an example of domain with two
interior high spots.

\section{Concluding remarks}

Two-dimensional sloshing is a common type of fluid oscillations in canals of uniform
cross-section and similar troughs with vertical end-walls. Investigating this motion
by means of a semi-inverse method, several examples of sloshing domains have been
constructed, all of which have the following common property: there is at least one
interior high spot, that is, a point on the free surface, where its extremal
elevation occurs for a fundamental eigenmode.

In the version of semi-inverse method applied here, the velocity potential satisfies
the free-surface boundary condition, involving the spectral parameter proportional
to the sloshing frequency squared, whereas a wetted contour is to be determined from
the no-flow condition. Using this condition, we restricted ourselves to considering
contours that are either adjacent to the $y$-axis or lying in the fourth quadrant.
Along with the latter contours, their symmetric images in the $y$-axis also provide
sloshing domains with interior high spots.

Let us recall some characteristic features of the constructed domains with interior
high spots:

$\bullet$ Many of these domains, but not all, have multiple interior high spots.

$\bullet$ All these domains are bulbous on the side, where an interior high spot is
located.

$\bullet$ Each found interior high spot is located close to an endpoint of the free
surface.

$\bullet$ The bottom profile of every found domain has at least one corner point (domains with smooth bottom profiles can also be constructed; see Remark~\ref{remark:4}).

$\bullet$ A single nodal line of the velocity potential connects $F$ and $B$ in each
example.

\smallskip

It is clear that a sloshing domain $W \subset \RR^2_-$ defines a trough $W \times
(0, \ell) \subset \RR^3_-$ for any $\ell > 0$. Moreover, if $u (x, y)$ is a
fundamental eigenmode of sloshing in~$W$, then this function plays the same role for
$W \times (0, \ell)$. Therefore, if $W$ has an interior high spot, then there is a
straight line (parallel to trough's generators) in the free surface of $W \times (0,
\ell)$, each point of which is a high spot interior with respect to the trough's
free surface.

In conclusion, we conjecture that the result obtained in \cite{KK1} for domains
satisfying John's condition and having smooth bottom (it bans interior high spots
for such domains) is still valid when the bottom is nonsmooth. The domain adjoining
$W_{5/2}$ on the left provides a basis for this conjecture as well as similar
domains adjoining $W_{7/2}$ and $W_{3}$.

\end{document}